\documentclass[a4paper,11pt]{article}

\usepackage{amsmath,amssymb,amsthm}
\usepackage{booktabs}
\usepackage{array}
\usepackage{hyperref}
\usepackage[margin=2.5cm]{geometry}

\newtheorem{theorem}{Theorem}[section]
\newtheorem{lemma}[theorem]{Lemma}
\newtheorem{proposition}[theorem]{Proposition}
\newtheorem{corollary}[theorem]{Corollary}
\newtheorem{definition}[theorem]{Definition}
\newtheorem{question}[theorem]{Question}

\theoremstyle{remark}
\newtheorem{remark}[theorem]{Remark}
\newtheorem{example}[theorem]{Example}

\newcommand{\ZZ}{\mathbb{Z}}

\newcommand{\QQ}{\mathbb{Q}}
\newcommand{\RR}{\mathbb{R}}
\newcommand{\Tr}{\operatorname{Tr}}
\newcommand{\Det}{\operatorname{Det}}

\newcommand{\GEN}{\textup{GEN}}
\newcommand{\PROP}{\textup{PROP}}
\newcommand{\KILL}{\textup{KILL}}

\begin{document}

\title{Cascade-free sequences, dispersion index, and state avoidance\\ for stateful digit-wise operations}

\author{Daniel Andreas Moj\\[4pt]
{\normalsize Alte Mainzer Str.\ 151, 55129 Mainz, Germany}\\[2pt]
{\normalsize\texttt{daniel.moj.apo@gmail.com}}}

\date{March 3, 2026}

\maketitle

\begin{abstract}
We show that cascade-free counting from carry theory is a special case of a general transfer matrix construction. For any binary stateful digit-wise operation with \GEN/\PROP/\KILL{} decomposition, the number of cascade-free sequences of length~$L$ depends on only two parameters: the alphabet size~$N$ and the product $\mathfrak{d} = |\GEN|\cdot|\PROP|$. The resulting sequence satisfies $a(L) = N\,a(L{-}1) - \mathfrak{d}\,a(L{-}2)$ and equals a scaled Chebyshev polynomial of the second kind with coupling parameter $x = N/(2\sqrt{\mathfrak{d}}) \ge 1$. We instantiate this for digit-wise addition and doubling in base~$p$. For odd primes the exact relation $a_{\mathrm{carry}}(L) = p^L\,a_{\mathrm{dbl}}(L)$ holds. For $p = 3$ the cascade-free doubling count equals the Fibonacci bisection $F(2L{+}2)$ via $U_L(3/2) = F(2L{+}2)$ (OEIS~A001906); we are not aware of this interpretation in the existing literature.

We analyse the dispersion index $D = \mathrm{Var}(\nu)/\mathrm{E}[\nu]$ of the state count for uniformly distributed inputs. For symmetric chains ($g = k$) the Poisson transition $D_\infty = 1$ occurs at $\mu = 1/3$, corresponding to base~$3$ where the Fibonacci bisection appears. The finite Poisson transition point $\mu^*(L)$ decreases strictly to~$1/3$ with rate $1/(6L) + O(1/L^2)$.

We generalise to state spaces $|S| > 2$ via state avoidance. The restricted transfer matrix has dimension~$s{-}1$; the Chebyshev representation persists for $|S| = 3$.
\end{abstract}



\section{Introduction}\label{sec:intro}

Adding two numbers in a positional number system is inherently
sequential: at each digit position one must know the incoming carry
before the outgoing carry can be determined.  Whether a carry
propagates depends on the digits at the current position and on the
state inherited from the previous one.  This state dependence is
responsible for carry-propagation delays in hardware adders, for the
correlation structure analysed by Holte~\cite{holte} and
Diaconis--Fulman~\cite{diaconis,diaconis-fulman-2012}, and for the connection to the
$p$-adic valuation established by Kummer~\cite{kummer}.  See also Nakano and Sadahiro~\cite{nakano} for the distributional structure of carry propagation.

We call an arithmetic operation \emph{stateful} when its output depends
not only on the current digits but on a state trajectory
$\sigma_0,\sigma_1,\ldots,\sigma_L$ inherited from previous positions.
The present paper isolates the combinatorial core of stateful
digit-wise arithmetic---the \emph{cascade-free counting
problem}---and develops it in full generality.

\subsection{The counting problem}\label{sec:intro-counting}

Fix a finite alphabet~$X$ with $|X|=N$ and a binary state space
$S=\{0,1\}$.  Each symbol $x\in X$ acts on~$S$ by one of three
transition types: it can force the state to~$1$ (\GEN), preserve
the current state (\PROP), or reset the state to~$0$
(\KILL).  A sequence $x_1,\ldots,x_L\in X^L$ is
\emph{cascade-free} if no \PROP{} symbol ever receives the
state~$1$---equivalently, if no \GEN{} symbol is immediately
followed by a \PROP{} symbol.

The cascade-free condition is a forbidden-adjacency constraint, and
the transfer matrix method computes the number of admissible sequences
in the standard way.  Since the transfer matrix is $2\times 2$, the
count satisfies a second-order linear recurrence.  At this level the
problem is an instance of subshift enumeration
\cite{lind-marcus} and forbidden-pattern counting
\cite{goulden-jackson}; the enumerative machinery is entirely classical.
What has apparently not been noticed is that trace and determinant of the
transfer matrix equal~$N$ and $\mathfrak{d} = |\GEN|\cdot|\PROP|$,
so the cascade-free count depends on just two arithmetically meaningful
parameters.  Different operations sharing the same~$N$
and~$\mathfrak{d}$ produce the same count, and a single coupling
parameter $x = N/(2\sqrt{\mathfrak{d}})$ governs the behaviour.

\subsection{Main results}\label{sec:intro-results}

The $2\times 2$~transfer matrix that governs cascade-free counting has
trace~$N$ and determinant~$\mathfrak{d}=|\GEN|\cdot|\PROP|$, and
these two invariants carry direct arithmetic meaning.  The cascade-free
count depends on the individual sizes of \GEN, \PROP, \KILL{} only
through~$N$ and~$\mathfrak{d}$.  The following is proved in
Sections~\ref{sec:counting} and~\ref{sec:chebyshev}.

\begin{theorem}[Universality and Chebyshev representation; Theorems~\ref{thm:main} and~\ref{thm:chebyshev}]\label{thm:intro-main}
Let $|X| = N$ and $\mathfrak{d} = |\GEN|\cdot|\PROP|$.  Then:

\emph{(a)} The cascade-free count satisfies $a(L) = N\,a(L{-}1) - \mathfrak{d}\,a(L{-}2)$ with $a(0) = 1$, $a(1) = N$.

\emph{(b)} Two operations with the same values of~$N$ and~$\mathfrak{d}$ produce identical cascade-free sequences.

\emph{(c)} For $\mathfrak{d} > 0$, the count equals $a(L) = (\sqrt{\mathfrak{d}}\,)^L\,U_L(x)$, where $U_L$ is the Chebyshev polynomial of the second kind and $x = N/(2\sqrt{\mathfrak{d}}) \ge 1$.
\end{theorem}

This universality has a concrete consequence for arithmetic:
digit-wise addition and digit-wise doubling in the same odd-prime
base~$p$ share the same coupling parameter, and their cascade-free
counts are related by an exact scaling law.

\begin{theorem}[Scaling law and Fibonacci specialisation; Theorems~\ref{thm:scaling} and~\ref{thm:fibonacci}]\label{thm:intro-scaling}
\emph{(a)} For all odd primes~$p$ and all $L \ge 0$,
$a_{\mathrm{carry}}(L) = p^L \cdot a_{\mathrm{dbl}}(L)$.

\emph{(b)} For $p = 3$, the cascade-free doubling count equals the Fibonacci bisection: $a_{\mathrm{dbl}}(L) = F(2L{+}2) = U_L(3/2)$.
\end{theorem}

OEIS~A001906~\cite{oeis-A001906} lists many interpretations of the Fibonacci bisection;
we are not aware of one arising from carry propagation or stateful
digit-wise operations.

The coupling parameter $x = 3/2$ appears again in the context of dispersion analysis.
For symmetric chains ($g = k$), the dispersion index of the state
count $\nu=\sum\sigma_k$ equals~$1$ (Poisson behaviour) precisely
at $\mu=t/N=1/3$, which corresponds to the same Chebyshev point
$x = 3/2$ (Theorem~\ref{thm:poisson-coupling}).  The finite Poisson
transition point $\mu^*(L)$ approaches~$1/3$ strictly from above,
with rate $1/(6L) + O(1/L^2)$ (Theorem~\ref{thm:monotonicity}).
That the Fibonacci point and the Poisson point coincide is not a
numerical accident but follows from the algebra of the coupling
parameter; whether it reflects something deeper remains open
(Question~\ref{q:poisson-fibonacci}).

\subsection{Related work}\label{sec:intro-related}

Carries as a stochastic process go back to
Holte~\cite{holte}, who modelled base-$p$ addition as a Markov chain on $\{0,1\}$. Diaconis and Fulman~\cite{diaconis,diaconis-fulman-2012} uncovered a striking connection between carry chains and card shuffling via the descent algebra;
Kummer's classical theorem~\cite{kummer} provides the $p$-adic side of the story by relating the $p$-adic valuation of binomial coefficients to carry counts.  Distributional aspects from a different angle are treated by Nakano and Sadahiro~\cite{nakano}.

For forbidden-pattern counting via transfer matrices we refer to Lind and Marcus~\cite{lind-marcus} (symbolic dynamics viewpoint) and
Goulden and Jackson~\cite{goulden-jackson} (cluster decomposition); Stanley~\cite{stanley} gives a textbook account.  The Chebyshev polynomials that arise from $2\times 2$ transfer matrices are treated in Mason and
Handscomb~\cite{mason} and Rivlin~\cite{rivlin}.

The dispersion index---called Fano factor in physics~\cite{fano} and index of dispersion in statistics~\cite{cox-lewis}---quantifies departure from Poisson behaviour. Variance formulas for Markov-dependent binary sequences in terms of the second eigenvalue are classical~\cite{kemeny-snell,billingsley}. What seems to be new here is the link to the carry-theoretic coupling parameter and the monotonicity proof for the finite Poisson transition point.

We do not treat the algebraic aspects connecting carry sequences to symmetric function theory~\cite{diaconis}, nor the deeper $p$-adic analysis beyond the basic Kummer interpretation.

\subsection{Outline}

Section~\ref{sec:prelim} fixes notation and recalls the needed background. Sections~\ref{sec:counting} and~\ref{sec:chebyshev} develop the counting framework: the universality theorem and the Chebyshev representation.  The arithmetic instances---carry propagation, doubling, and the Fibonacci link---occupy Section~\ref{sec:instances}.  The Markov chain and variance analysis follow in Section~\ref{sec:markov}, and the Poisson transition with its monotonicity proof in Section~\ref{sec:poisson}.  Section~\ref{sec:generalisation} extends the theory to $|S| > 2$.

\subsection{Notation}\label{sec:notation}

Table~\ref{tab:notation} collects the principal notation used throughout the paper.

\begin{table}[ht]
\centering
\caption{Principal notation.}
\label{tab:notation}
\begin{tabular}{cl}
\toprule
Symbol & Meaning \\
\midrule
$X,\; N = |X|$ & Finite alphabet and its size \\
$S = \{0,1\}$ & Binary state space \\
$g,\; t,\; k$ & Sizes of the \GEN, \PROP, \KILL{} classes ($g + t + k = N$) \\
$\mathfrak{d} = g\cdot t$ & Determinant of the transfer matrix \\
$T$ & $2 \times 2$ transfer matrix for cascade-free counting \\
$a(L)$ & Number of cascade-free sequences of length~$L$ \\
$x = N/(2\sqrt{\mathfrak{d}})$ & Coupling parameter \\
$U_L(x)$ & Chebyshev polynomial of the second kind \\
$M$ & Markov transition matrix for the state chain \\
$\mu = t/N$ & Second eigenvalue of~$M$ \\
$\pi_0,\, \pi_1$ & Stationary distribution of the state chain \\
$\nu = \sum_{k=1}^L \sigma_k$ & Total state count \\
$D = \mathrm{Var}(\nu)/\mathrm{E}[\nu]$ & Dispersion index \\
$D_\infty$ & Asymptotic dispersion index ($L \to \infty$) \\
$\mu^*(L)$ & Finite Poisson transition point ($D = 1$) \\
$\widetilde{T}$ & Restricted transfer matrix for state avoidance ($|S| > 2$) \\
\bottomrule
\end{tabular}
\end{table}

\section{Preliminaries}\label{sec:prelim}

We collect the definitions and standard results needed in the sequel.

\subsection{Chebyshev polynomials of the second kind}\label{sec:prelim-cheb}

The Chebyshev polynomials of the second kind $U_n\colon \RR \to \RR$ are defined by the recurrence
\begin{equation}\label{eq:cheb-recurrence}
U_n(x) = 2x\,U_{n-1}(x) - U_{n-2}(x), \quad U_0(x) = 1,\; U_1(x) = 2x.
\end{equation}
For $|x| > 1$, the closed form is
\begin{equation}\label{eq:cheb-closed}
U_n(x) = \frac{\alpha^{n+1} - \beta^{n+1}}{\alpha - \beta}, \qquad \alpha = x + \sqrt{x^2 - 1},\quad \beta = x - \sqrt{x^2 - 1},
\end{equation}
where $\alpha\beta = 1$ and $\alpha + \beta = 2x$. At $x = 1$ the polynomial evaluates to $U_n(1) = n + 1$. The generating function is
\begin{equation}\label{eq:cheb-gf}
\sum_{n \ge 0} U_n(x)\,z^n = \frac{1}{1 - 2xz + z^2}.
\end{equation}
Standard references for these properties include Mason and Handscomb~\cite{mason} and Rivlin~\cite{rivlin}.

\subsection{Transfer matrices and forbidden adjacencies}\label{sec:prelim-tm}

Let $\Sigma$ be a finite alphabet and $\mathcal{F} \subset \Sigma \times \Sigma$ a set of forbidden pairs. A sequence $x_1,\ldots,x_L \in \Sigma^L$ is \emph{admissible} if $(x_k, x_{k+1}) \notin \mathcal{F}$ for all $1 \le k \le L{-}1$. The number of admissible sequences of length~$L$ is determined by a transfer matrix $A$ with $A_{ij} = 1$ if $(i,j) \notin \mathcal{F}$ and $A_{ij} = 0$ otherwise: the count equals $v_0^{\mathsf{T}} A^{L-1}\, \mathbf{1}$ for a suitable initial vector $v_0$.

When the forbidden structure can be encoded by a matrix with block structure, the effective transfer matrix has smaller dimension. In our setting, the state trajectory $\sigma_0,\ldots,\sigma_L$ is binary, and the forbidden condition (no \GEN{} followed by \PROP) gives rise to a $2 \times 2$ transfer matrix. The admissible sequence count then satisfies a linear recurrence of order~$2$ with coefficients determined by the trace and determinant of this matrix. This is a standard construction; see Stanley~\cite{stanley}, Chapter~4, and Lind and Marcus~\cite{lind-marcus}, Chapter~4.

\subsection{Dispersion index}\label{sec:prelim-disp}

For a non-negative integer-valued random variable~$\nu$ with $\mathrm{E}[\nu] > 0$, the \emph{dispersion index} (or index of dispersion) is
\[
D = \frac{\mathrm{Var}(\nu)}{\mathrm{E}[\nu]}.
\]
When $\nu$ follows a Poisson distribution, $D = 1$. Values $D > 1$ indicate overdispersion (clustering), while $D < 1$ indicates underdispersion (more regular than Poisson). For sums of correlated Bernoulli variables arising from a Markov chain, the dispersion index depends on the autocorrelation structure and can be computed explicitly in terms of the second eigenvalue of the transition matrix; see Cox and Lewis~\cite{cox-lewis} and Fano~\cite{fano}.

\section{The counting framework}\label{sec:counting}

\subsection{Stateful digit-wise operations}

Let $p \ge 2$ be a base and $S = \{0,1\}$ a binary state space. A \emph{stateful digit-wise operation} processes a sequence of input symbols $x_1, x_2, \ldots, x_L$ from a finite alphabet~$X$ ($|X| = N$) and produces a state trajectory $\sigma_0, \sigma_1, \ldots, \sigma_L$ with $\sigma_0 = 0$.

At each position~$k$, the symbol~$x_k$ together with~$\sigma_{k-1}$ determines the successor state~$\sigma_k$. The state transition function $T_x\colon S \to S$ assigns to each symbol~$x$ a map $\{0,1\} \to \{0,1\}$.

\subsection{The \GEN/\PROP/\KILL{} decomposition}

Since $|S| = 2$, there are exactly four possible maps $\{0,1\} \to \{0,1\}$: two constant maps, the identity, and the negation. We restrict to operations without negation (no ANTI transitions). Each symbol $x \in X$ falls into exactly one of three classes, whose sizes we denote $g$, $t$, $k$ respectively.

\begin{definition}[\GEN/\PROP/\KILL{} classification]\label{def:gpk}
A symbol~$x$ is called \GEN{} (generation) if $T_x(\sigma) = 1$ for all~$\sigma$, \PROP{} (propagation) if $T_x(\sigma) = \sigma$ for all~$\sigma$, and \KILL{} (annihilation) if $T_x(\sigma) = 0$ for all~$\sigma$.  We write $g = |\GEN|$, $t = |\PROP|$, $k = |\KILL|$, so that $g + t + k = N$.
\end{definition}

\subsection{Cascade-free sequences}

\begin{definition}\label{def:cascade-free}
A sequence $x_1,\ldots,x_L \in X^L$ is called \emph{cascade-free} if $P^*(x) = 0$, i.e.\ no \PROP{} symbol in the sequence receives the state $\sigma = 1$.
\end{definition}

\begin{proposition}[Equivalent characterisation]\label{prop:equiv}
A sequence is cascade-free if and only if no \GEN{} symbol is immediately followed by a \PROP{} symbol.
\end{proposition}

\begin{proof}
($\Rightarrow$) Suppose \GEN{} at position~$j$ is directly followed by \PROP{} at position~$j{+}1$. Since \GEN{} forces $\sigma_j = 1$ regardless of the input state, the \PROP{} symbol at~$j{+}1$ receives $\sigma_j = 1$ and preserves it. Hence $P^* \ge 1$, contradicting the cascade-free assumption.

($\Leftarrow$) Suppose no $\GEN\to\PROP$ adjacency exists. We show by induction on~$k$ that the state before every \PROP{} symbol is~$0$. The base is $\sigma_0 = 0$. For the inductive step, consider position~$k$ with $x_k \in \PROP$. The predecessor $x_{k-1}$ is either \KILL{} or \PROP{} (since $\GEN\to\PROP$ is excluded). If $x_{k-1} \in \KILL$, then $\sigma_{k-1} = 0$ regardless of~$\sigma_{k-2}$. If $x_{k-1} \in \PROP$, then $\sigma_{k-1} = \sigma_{k-2}$, and by the inductive hypothesis $\sigma_{k-2} = 0$. In both cases $\sigma_{k-1} = 0$, so the \PROP{} at position~$k$ receives state~$0$ and $P^* = 0$.
\end{proof}

\subsection{Transfer matrix construction}\label{sec:transfer}

We now set up the transfer matrix that counts cascade-free sequences. The key idea is to classify each position by the state it produces, and to track which transitions to the next position are forbidden.

Partition the alphabet into two groups based on the successor state they create: a position occupied by a \KILL{} or \PROP{} symbol leaves the system in a ``rest'' state (denoted~R), while a \GEN{} symbol puts the system in a ``generating'' state (denoted~G). From state~R, every successor symbol is allowed; from state~G, only \KILL{} and \GEN{} are allowed (since \PROP{} is forbidden after \GEN). This gives the $2 \times 2$ transfer matrix (cf.~\cite{stanley}, Chapter~4)
\begin{equation}\label{eq:transfer}
T = \begin{pmatrix} k+t & g \\ k & g \end{pmatrix},
\end{equation}
where the rows and columns are indexed by (R, G).

\subsection{Invariants of the transfer matrix}

\begin{lemma}\label{obs:invariants}
$\Tr(T) = N$ and $\Det(T) = tg =: \mathfrak{d}$. Both depend only on~$N$ and~$\mathfrak{d}$, not on the individual values of~$g$, $t$, $k$.
\end{lemma}

\begin{proof}
Direct computation: $\Tr(T) = (k + t) + g = N$ and $\Det(T) = (k+t)g - gk = tg = \mathfrak{d}$.
\end{proof}

\subsection{Counting theorem}

\begin{theorem}[Cascade-free counting]\label{thm:main}
Let $X$ be an alphabet with $|X| = N$ and \GEN/\PROP/\KILL{} decomposition with $|\GEN|\cdot|\PROP| = \mathfrak{d}$. The number of cascade-free sequences $a(L)$ in $X^L$ satisfies
\begin{equation}\label{eq:recurrence}
a(L) = N\,a(L{-}1) - \mathfrak{d}\,a(L{-}2) \quad\text{for } L \ge 2,
\end{equation}
with $a(0) = 1$, $a(1) = N$.

\emph{Case~1} (simple eigenvalues, $N^2 > 4\mathfrak{d}$). The eigenvalues of~$T$ are
\begin{equation}\label{eq:eigenvalues}
\lambda_{1,2} = \frac{N \pm \sqrt{N^2 - 4\mathfrak{d}}}{2},
\end{equation}
and the closed form is $a(L) = (\lambda_1^{L+1} - \lambda_2^{L+1})/(\lambda_1 - \lambda_2)$.

\emph{Case~2} (double eigenvalue, $N^2 = 4\mathfrak{d}$). This occurs if and only if $k = 0$ and $g = t = N/2$. Then $\lambda_1 = \lambda_2 = N/2$ and $a(L) = (L{+}1)(N/2)^L$.
\end{theorem}

\begin{proof}
We first establish the initial conditions and then derive the recurrence from the transfer matrix.

The initial conditions are $a(0) = 1$ (the empty sequence is vacuously cascade-free) and $a(1) = N$ (a single-symbol sequence has no adjacent pair, so every symbol is admissible).

For $L \ge 2$, the cascade-free count is computed as
\[
a(L) = v_0^{\mathsf{T}} T^{L-1} \mathbf{1},
\]
where $v_0 = (k{+}t,\, g)^{\mathsf{T}}$ encodes the number of admissible first symbols entering each state and $\mathbf{1} = (1,1)^{\mathsf{T}}$. The characteristic polynomial of~$T$ is $\lambda^2 - N\lambda + \mathfrak{d} = 0$ by Lemma~\ref{obs:invariants}. By the Cayley--Hamilton theorem, $T^2 = NT - \mathfrak{d}I$, which implies that the sequence $a(L) = v_0^{\mathsf{T}} T^{L-1} \mathbf{1}$ satisfies the recurrence~\eqref{eq:recurrence}.

In Case~1, the general solution of~\eqref{eq:recurrence} is $a(L) = A\lambda_1^L + B\lambda_2^L$ for constants $A$, $B$. From $a(0) = A + B = 1$ and $a(1) = A\lambda_1 + B\lambda_2 = N = \lambda_1 + \lambda_2$, we obtain $A = \lambda_1/(\lambda_1 - \lambda_2)$ and $B = -\lambda_2/(\lambda_1 - \lambda_2)$, yielding the stated closed form.

In Case~2, the discriminant $N^2 - 4\mathfrak{d} = 0$ gives the double eigenvalue $\lambda = N/2$. Since $g + t + k = N$ and $gt = N^2/4$, the AM-GM inequality yields $g = t$ and $k = 0$, so $g = t = N/2$. The general solution is $a(L) = (c_1 + c_2 L)\lambda^L$. From $a(0) = c_1 = 1$ and $a(1) = (1 + c_2)\lambda = 2\lambda = N$, we get $c_2 = 1$, completing the proof.
\end{proof}

\begin{corollary}[Universality]\label{cor:universality}
Two distinct stateful digit-wise operations with the same alphabet size~$N$ and the same product~$\mathfrak{d}$ produce identical cascade-free sequences, regardless of how $g$, $t$, and $k$ are individually distributed.
\end{corollary}

\subsection{Parametrisation}

The cascade-free sequences form a two-parameter family $\{a_{N,d}\}_{N \ge 1,\, 0 \le d \le \lfloor N^2/4\rfloor}$ with $d = \mathfrak{d}$. At the boundary $d = 0$ there is no state propagation and $a(L) = N^L$; at $d = \lfloor N^2/4 \rfloor$ the growth is slowest, $a(L) = (L{+}1)(N/2)^L$, with maximal state interaction. The inequality $\mathfrak{d} \le N^2/4$ follows from AM-GM applied to $g + t \le N$ with $gt = \mathfrak{d}$.

\section{Chebyshev polynomials and the coupling parameter}\label{sec:chebyshev}

\subsection{The Chebyshev representation}

\begin{theorem}[Chebyshev representation]\label{thm:chebyshev}
For $\mathfrak{d} > 0$,
\begin{equation}\label{eq:chebyshev}
a(L) = (\sqrt{\mathfrak{d}}\,)^L \cdot U_L(x), \quad\text{where } x = \frac{N}{2\sqrt{\mathfrak{d}}}.
\end{equation}
It holds that $x \ge 1$, with equality if and only if $k = 0$ and $g = t = N/2$.
\end{theorem}

\begin{proof}
Rescale by setting $b(L) = a(L)/(\sqrt{\mathfrak{d}}\,)^L$. Dividing the recurrence~\eqref{eq:recurrence} by $(\sqrt{\mathfrak{d}}\,)^L$ yields
\[
b(L) = \frac{N}{\sqrt{\mathfrak{d}}}\,b(L{-}1) - b(L{-}2) = 2x\,b(L{-}1) - b(L{-}2),
\]
with initial conditions $b(0) = 1 = U_0(x)$ and $b(1) = N/\sqrt{\mathfrak{d}} = 2x = U_1(x)$. Since $b$ satisfies the same recurrence~\eqref{eq:cheb-recurrence} with the same initial conditions as $U_L(x)$, we conclude $b(L) = U_L(x)$ for all $L \ge 0$.

The inequality $x \ge 1$ is equivalent to $N^2 \ge 4\mathfrak{d} = 4gt$. Since $g + t \le g + t + k = N$, the AM-GM inequality gives $gt \le (g+t)^2/4 \le N^2/4$, with equality if and only if $g = t$ and $k = 0$.
\end{proof}

\subsection{The coupling parameter}

We call $x = N/(2\sqrt{\mathfrak{d}})$ the \emph{coupling parameter} of the stateful operation. It measures how strongly the $\GEN$--$\PROP$ coupling suppresses the cascade-free count relative to the unrestricted count~$N^L$.

\begin{remark}[Limiting behaviour]\label{rem:limiting}
For $x \to \infty$ (weak coupling), $a(L) \sim N^L$. For $x = 1$ (maximal coupling), $a(L) = (N/2)^L(L{+}1)$. The rate drops from~$N$ to~$N/2$ with a polynomial correction.  The cascade-free density $\rho(L) = a(L)/N^L$ decays exponentially at rate $-\ln(\lambda_1/N)$.
\end{remark}

\subsection{Geometric interpretation}

Corollary~\ref{cor:universality} has a geometric explanation: all stateful operations with the same coupling parameter~$x$ lie on the same Chebyshev curve $L \mapsto U_L(x)$. The family of all cascade-free sequences becomes a one-dimensional object, parametrised by $x \in [1,\infty)$ alone.

\subsection{Generating function}

The generating function is
\begin{equation}\label{eq:gf}
\sum_{L \ge 0} a(L)\,z^L = \frac{1}{1 - Nz + \mathfrak{d}\,z^2},
\end{equation}
rational with degree-$2$ denominator. This follows from the Chebyshev generating function~\eqref{eq:cheb-gf} after the substitution $z \mapsto \sqrt{\mathfrak{d}}\,z$.

\begin{corollary}
Every OEIS sequence with generating function $1/(1 - Nz + dz^2)$ for integers $N \ge 1$, $d \ge 0$ is a cascade-free sequence for a stateful digit-wise operation with parameters $(N, \mathfrak{d} = d)$.
\end{corollary}

\begin{question}\label{q:oeis}
Which further carry or doubling sequences arising from specific arithmetic operations coincide with known OEIS entries via the generating function $1/(1 - Nz + dz^2)$?
\end{question}

\subsection{The discriminant structure}\label{sec:discriminant}

The discriminants for odd~$p$:
\begin{align*}
\Delta_{\mathrm{carry}} &= p^2\bigl((p{-}1)^2 + 1\bigr), \\
\Delta_{\mathrm{dbl}} &= (p{-}1)^2 + 1.
\end{align*}
Hence $\Delta_{\mathrm{carry}} = p^2 \cdot \Delta_{\mathrm{dbl}}$ and $\lambda_{\mathrm{carry}} = p\,\lambda_{\mathrm{dbl}}$, consistent with Theorem~\ref{thm:scaling}.

The eigenvalues are rational if and only if $(p{-}1)^2 + 1$ is a perfect square. Setting $m^2 = (p{-}1)^2 + 1$ gives $m^2 - (p{-}1)^2 = 1$, hence $(m{-}(p{-}1))(m{+}(p{-}1)) = 1$. Since both factors are positive integers, both equal~$1$, forcing $p{-}1 = 0$. The eigenvalues are therefore irrational for every $p \ge 2$.

\begin{table}[ht]
\centering
\caption{Discriminant spectrum.}
\label{tab:discriminant}
\begin{tabular}{ccccc}
\toprule
$p$ & $\Delta_{\mathrm{dbl}}$ & Factorisation & Nature & $x$ \\
\midrule
2 & 2 & 2 & $\sqrt{2}$ irrational & $\sqrt{2} \approx 1.414$ \\
3 & 5 & 5 & $\sqrt{5}$, golden ratio & $3/2 = 1.500$ \\
5 & 17 & 17 & Fermat prime & $5/\sqrt{8} \approx 1.768$ \\
7 & 37 & 37 & prime & $7/\sqrt{12} \approx 2.021$ \\
13 & 145 & $5 \cdot 29$ & composite & $13/\sqrt{24} \approx 2.653$ \\
\bottomrule
\end{tabular}
\end{table}

The coupling parameter grows as $\sqrt{p/2}$ for large~$p$. The ``strong coupling'' regime ($x < 2$) covers only $p \in \{2,3,5\}$.

\begin{question}\label{q:bunyakovsky}
Is $(p{-}1)^2 + 1$ prime for infinitely many primes~$p$?  This is a special case of the Bunyakovsky conjecture; see \cite{ribenboim}.
\end{question}

\section{Arithmetic instances}\label{sec:instances}

\subsection{Carry propagation in addition}\label{sec:carry}

For digit-wise addition of two numbers in base~$p$, the alphabet is $X = \{0,\ldots,p{-}1\}^2$, so $N = p^2$. A digit pair $(a,b)$ with incoming carry $\sigma \in \{0,1\}$ produces sum $a + b + \sigma$ and outgoing carry $\lfloor(a+b+\sigma)/p\rfloor$. The \GEN/\PROP/\KILL{} decomposition classifies pairs by whether $a + b \ge p$ (carry regardless of input: \GEN), $a + b = p - 1$ (carry if and only if $\sigma = 1$: \PROP), or $a + b \le p - 2$ (no carry regardless: \KILL). Counting the pairs in each class yields $g = p(p{-}1)/2$, $t = p$, $k = p(p{-}1)/2$, and $\mathfrak{d} = p^2(p{-}1)/2$.

The coupling parameter is $x_{\mathrm{carry}} = p/\sqrt{2(p{-}1)}$.

Theorem~\ref{thm:main} with these parameters gives the recurrence $a(L) = p^2\,a(L{-}1) - p^2(p{-}1)/2\cdot a(L{-}2)$. The sequence for $p = 2$ is $1,4,14,48,164,560,\ldots$ (OEIS~A007070~\cite{oeis-A007070}).

\subsection{Carry propagation in doubling}\label{sec:doubling}

Doubling $2m$ of a number~$m$ in base~$p$ processes single digits, so the alphabet is $X = \{0,\ldots,p{-}1\}$ with $N = p$. A digit $d$ with incoming carry $\sigma \in \{0,1\}$ produces $2d + \sigma$ with outgoing carry $\lfloor(2d + \sigma)/p\rfloor$. For odd~$p$: a digit~$d$ satisfies $2d \ge p$ (i.e.\ $d \ge (p+1)/2$: \GEN) for $(p{-}1)/2$ values, $2d = p - 1$ (i.e.\ $d = (p{-}1)/2$: \PROP) for exactly one value, and $2d \le p - 2$ (\KILL) for $(p{-}1)/2$ values. Hence $g = (p{-}1)/2$, $t = 1$, $k = (p{-}1)/2$, $\mathfrak{d} = (p{-}1)/2$.

The coupling parameter is $x_{\mathrm{dbl}} = p/\sqrt{2(p{-}1)} = x_{\mathrm{carry}}$.

\begin{theorem}[Cascade-free doubling]\label{thm:doubling}
For odd~$p$, the number $a_{\mathrm{dbl}}(L)$ of $L$-digit numbers $m \in [0, p^L)$ whose doubling is cascade-free satisfies
\[
a_{\mathrm{dbl}}(L) = p\,a_{\mathrm{dbl}}(L{-}1) - \tfrac{p{-}1}{2}\,a_{\mathrm{dbl}}(L{-}2),
\]
with eigenvalues $\lambda_{1,2} = (p \pm \sqrt{(p{-}1)^2 + 1})/2$. For $p = 2$, $a_{\mathrm{dbl}}(L) = 2^L$ (trivial).
\end{theorem}

\subsection{The carry--doubling scaling law}

\begin{theorem}[Carry--doubling relation]\label{thm:scaling}
For all odd primes~$p$ and all $L \ge 0$,
\begin{equation}\label{eq:scaling}
a_{\mathrm{carry}}(L) = p^L \cdot a_{\mathrm{dbl}}(L).
\end{equation}
For $p = 2$ the scaling law fails because the \GEN/\PROP/\KILL{} proportions differ.
\end{theorem}

\begin{proof}
Set $b(L) = p^L a_{\mathrm{dbl}}(L)$ and check that $b$ satisfies the carry recurrence.  Using the doubling recurrence from Theorem~\ref{thm:doubling}:
\begin{align*}
b(L) &= p^L\bigl[p\,a_{\mathrm{dbl}}(L{-}1) - \tfrac{p{-}1}{2}\,a_{\mathrm{dbl}}(L{-}2)\bigr] \\
&= p \cdot p^{L-1} a_{\mathrm{dbl}}(L{-}1) \cdot p - \tfrac{p{-}1}{2}\cdot p^{L-2}a_{\mathrm{dbl}}(L{-}2) \cdot p^2 \\
&= p^2 b(L{-}1) - \tfrac{p^2(p{-}1)}{2}\,b(L{-}2).
\end{align*}
This is exactly the carry recurrence (Section~\ref{sec:carry}) with $\mathfrak{d}_{\mathrm{carry}} = p^2(p{-}1)/2$. The initial conditions match: $b(0) = 1 = a_{\mathrm{carry}}(0)$ and $b(1) = p \cdot p = p^2 = a_{\mathrm{carry}}(1)$. Since both sequences satisfy the same recurrence with the same initial conditions, they are identical.
\end{proof}

\begin{remark}
The equality of densities $a_{\mathrm{carry}}(L)/p^{2L} = a_{\mathrm{dbl}}(L)/p^L$ follows from the identical \GEN/\PROP/\KILL{} proportions: $g/N = (p{-}1)/(2p)$, $t/N = 1/p$, $k/N = (p{-}1)/(2p)$ in both cases.
\end{remark}

\begin{question}\label{q:scaling-general}
For which other pairs of stateful operations---one acting on pairs, one on single elements---does an exact scaling law of the form $a_1(L) = c^L\,a_2(L)$ hold?
\end{question}

\subsection{The Fibonacci connection}\label{sec:fibonacci}

For $p = 3$ the doubling recurrence becomes $a(L) = 3a(L{-}1) - a(L{-}2)$.

\begin{theorem}[Fibonacci bisection]\label{thm:fibonacci}
For $p = 3$,
\[
a_{\mathrm{dbl}}(L) = F(2L{+}2) = U_L(3/2),
\]
where $F$ denotes the Fibonacci sequence.
\end{theorem}

\begin{proof}
We give two arguments, one via recurrences and one via the Chebyshev--Binet connection.

\emph{First proof (recurrence comparison).} The Fibonacci bisection $b(L) = F(2L{+}2)$ satisfies $b(L) = 3b(L{-}1) - b(L{-}2)$, which is a standard identity following from $F(2n{+}2) = 3F(2n) - F(2n{-}2)$ (see e.g.\ \cite{koshy,vajda}). The initial conditions $b(0) = F(2) = 1$ and $b(1) = F(4) = 3$ match $a_{\mathrm{dbl}}(0) = 1$ and $a_{\mathrm{dbl}}(1) = 3$. Since both sequences satisfy the same second-order recurrence with the same initial conditions, they are equal.

\emph{Second proof (Chebyshev--Binet).} For $p = 3$ we have $\mathfrak{d} = 1$ and $x = 3/2$. The closed form~\eqref{eq:cheb-closed} gives $U_L(3/2) = (\alpha^{L+1} - \beta^{L+1})/(\alpha - \beta)$ with $\alpha = (3{+}\sqrt{5})/2 = \varphi^2$ and $\beta = (3{-}\sqrt{5})/2 = \psi^2$, where $\varphi = (1{+}\sqrt{5})/2$ is the golden ratio and $\psi = (1{-}\sqrt{5})/2$. Since $\alpha - \beta = \sqrt{5}$, this yields $U_L(3/2) = (\varphi^{2L+2} - \psi^{2L+2})/\sqrt{5} = F(2L{+}2)$ by Binet's formula.
\end{proof}

\begin{corollary}
The number of $3$-adic numbers $m \in [0, 3^L)$ whose doubling produces no carry cascade is the $(2L{+}2)$-nd Fibonacci number.
\end{corollary}

Among the many known interpretations of OEIS~A001906~\cite{oeis-A001906}, we are not aware of one arising from carry propagation or stateful digit-wise operations.

\begin{table}[ht]
\centering
\caption{Cascade-free doubling counts for $p = 3$.}
\label{tab:fibonacci}
\begin{tabular}{ccccc}
\toprule
$L$ & $a_{\mathrm{dbl}}(L)$ & Fibonacci & $U_L(3/2)$ & $a_{\mathrm{carry}}(L)$ \\
\midrule
0 & 1 & $F(2) = 1$ & 1 & 1 \\
1 & 3 & $F(4) = 3$ & 3 & 9 \\
2 & 8 & $F(6) = 8$ & 8 & 72 \\
3 & 21 & $F(8) = 21$ & 21 & 567 \\
4 & 55 & $F(10) = 55$ & 55 & 4455 \\
5 & 144 & $F(12) = 144$ & 144 & 34992 \\
\bottomrule
\end{tabular}
\end{table}

\subsection{The distinguished role of base $p = 3$}\label{sec:why-p3}

The condition $\mathfrak{d} = 1$ holds precisely for $p = 3$ among the primes. It means there is exactly one \PROP{} symbol and one \GEN{} symbol. A sequence in which these two symbols never appear consecutively is cascade-free---the classical problem of non-adjacent forbidden pairs, whose solution is the Fibonacci numbers.

The golden ratio enters because the eigenvalues of the recurrence $\lambda^2 - 3\lambda + 1 = 0$ are $\lambda_{1,2} = \varphi^2, \psi^2$, where $\varphi = (1{+}\sqrt{5})/2$ is the golden ratio. This is a direct consequence of $\varphi^2 = \varphi + 1$: squaring yields $\varphi^4 = 3\varphi^2 - 1$, so $\varphi^2$ satisfies the characteristic polynomial. The Fibonacci bisection $F(2L{+}2)$ then follows from the Binet representation with base $\varphi^2$ instead of $\varphi$.

\subsection{Extensions to $p \ge 5$}

For $p = 5$: $a(L) = 5a(L{-}1) - 2a(L{-}2)$, with eigenvalues $(5 \pm \sqrt{17})/2$.

For $p = 7$: $a(L) = 7a(L{-}1) - 3a(L{-}2)$, with eigenvalues $(7 \pm \sqrt{37})/2$.

\section{Markov chain and variance analysis}\label{sec:markov}

\subsection{Decomposition into local and stateful components}

For any stateful operation the state count decomposes as
\[
|\{k : \sigma_k = 1\}| = G(x) + P^*(x),
\]
where $G(x) = |\{k : x_k \in \GEN\}|$ is the local component and $P^*(x)$ is the stateful component (\PROP{} symbols receiving $\sigma = 1$). For addition in base~$p$, the local component counts positions where $a_k + b_k \ge p$ (digits that generate a carry independently of the incoming state), while the stateful component counts propagated carries. This decomposition refines the classical Kummer theorem~\cite{kummer}: the $p$-adic valuation $\nu_p\binom{m+n}{m} = G + P^*$.

\subsection{Markov chain and stationary distribution}

For uniformly distributed symbols the state sequence $(\sigma_k)$ forms a Markov chain on $\{0,1\}$ (cf.~\cite{kemeny-snell,seneta}) with transition matrix
\[
M = \begin{pmatrix} (t{+}k)/N & g/N \\ k/N & (g{+}t)/N \end{pmatrix}.
\]

\begin{proposition}[Stationary distribution]\label{prop:stationary}
The stationary distribution is $\pi(0) = k/(g{+}k)$, $\pi(1) = g/(g{+}k)$. The chain is doubly stochastic ($\pi = (1/2,1/2)$) if and only if $g = k$.
\end{proposition}

\subsection{Autocorrelation}

The second eigenvalue of the transition matrix is $\mu = t/N$.

\begin{proposition}\label{prop:autocorrelation}
$\mathrm{Corr}(\sigma_j, \sigma_{j+m}) = (t/N)^m$. The correlation length is $\xi = 1/\ln(N/t)$.
\end{proposition}

\subsection{Expected values}

\begin{theorem}[Stationary expected values]\label{thm:expectations}
In the stationary regime,
\[
\frac{\mathrm{E}[P^*]}{L} = \frac{tg}{N(g{+}k)}, \quad
\frac{\mathrm{E}[P^*]}{\mathrm{E}[\sigma {=} 1]} = \frac{t}{N}.
\]
For carries: $t/N = p/p^2 = 1/p$, consistent with the autocorrelation $\mathrm{Corr}(\sigma_j,\sigma_{j+m}) = p^{-m}$ observed in Proposition~\ref{prop:autocorrelation}.
\end{theorem}

\begin{proof}
In the stationary regime, $\sigma_k$ is Bernoulli with parameter~$\pi_1 = g/(g{+}k)$, so $\mathrm{E}[\sigma_k = 1] = L\pi_1$. The propagated-state count is $P^* = \sum_{k=1}^L \mathbf{1}[x_k \in \PROP,\, \sigma_{k-1} = 1]$. At each position, the symbol $x_k$ is drawn independently from~$X$, while $\sigma_{k-1}$ is determined by the Markov chain. Since the event $\{x_k \in \PROP\}$ depends only on $x_k$ and is independent of~$\sigma_{k-1}$ conditional on the chain:
\[
\mathrm{P}(x_k \in \PROP)\cdot\mathrm{P}(\sigma_{k-1} = 1) = \frac{t}{N}\cdot\pi_1.
\]
Summing over $k = 1,\ldots,L$:
\[
\mathrm{E}[P^*] = L\cdot\frac{t\pi_1}{N} = L\cdot\frac{tg}{N(g{+}k)}.
\]
The ratio follows:
\[
\frac{\mathrm{E}[P^*]}{\mathrm{E}[\sigma {=} 1]}
= \frac{L\cdot tg/\bigl(N(g{+}k)\bigr)}{L\cdot g/(g{+}k)}
= \frac{t}{N}. \qedhere
\]
\end{proof}

\subsection{Variance of the state count}\label{sec:variance}

Let $\nu = \sum_{k=1}^L \sigma_k$ denote the total number of positions with state~$1$.

\subsubsection*{Stationary regime}

\begin{proposition}[Stationary variance]\label{prop:var-stat}
Let $\sigma_0 \sim \pi$. Then
\begin{equation}\label{eq:var-stat}
\mathrm{Var}(\nu) = \pi_1\pi_0\biggl[\frac{L(1{+}\mu)}{1{-}\mu} - \frac{2\mu(1{-}\mu^L)}{(1{-}\mu)^2}\biggr].
\end{equation}
\end{proposition}

\begin{proof}
Each $\sigma_k$ is Bernoulli with parameter~$\pi_1$, so $\mathrm{Var}(\sigma_k) = \pi_1\pi_0$. The autocovariance is $\mathrm{Cov}(\sigma_j,\sigma_k) = \pi_1\pi_0\mu^{|k-j|}$ (see~\cite{kemeny-snell}). Expanding the variance of the sum:
\[
\mathrm{Var}(\nu) = \sum_{k=1}^L \mathrm{Var}(\sigma_k) + 2\sum_{1 \le j<k \le L}\mathrm{Cov}(\sigma_j,\sigma_k)
= \pi_1\pi_0\Bigl[L + 2\sum_{1 \le j<k \le L}\mu^{k-j}\Bigr].
\]
The double sum evaluates to $\sum_{m=1}^{L-1}(L-m)\mu^m = L\mu/(1{-}\mu) - \mu(1{-}\mu^L)/(1{-}\mu)^2$ via standard geometric series manipulations, yielding~\eqref{eq:var-stat}.
\end{proof}

\begin{corollary}[Asymptotic dispersion index]\label{cor:dispersion}
\[
D_\infty = \frac{\pi_0(1{+}\mu)}{1{-}\mu} = \frac{k(g{+}k{+}2t)}{(g{+}k)^2}.
\]
\end{corollary}

\subsubsection*{Transient regime ($\sigma_0 = 0$)}

\begin{proposition}\label{prop:transient}
Starting from $\sigma_0 = 0$:

\emph{(a)} $\mathrm{P}(\sigma_k = 1) = \pi_1(1 - \mu^k)$.

\emph{(b)} $\mathrm{E}[\nu] = \pi_1[L - \mu(1{-}\mu^L)/(1{-}\mu)]$.

\emph{(c)} For $1 \le j < k \le L$: $\mathrm{Cov}(\sigma_j,\sigma_k) = \pi_1(1{-}\mu^j)\mu^{k-j}(\pi_0 + \pi_1\mu^j)$.
\end{proposition}

\begin{proof}
The transition matrix~$M$ has eigenvalues~$1$ and~$\mu = t/N$ with spectral decomposition $M = \Pi + \mu\,R$, where $\Pi_{ij} = \pi_j$ and $R$ is the rank-one residual. This yields the matrix powers $(M^n)_{01} = \pi_1(1 - \mu^n)$ and $(M^n)_{11} = \pi_1 + \pi_0\mu^n$.

\emph{Part~(a).} $\mathrm{P}(\sigma_k = 1 \mid \sigma_0 = 0) = (M^k)_{01} = \pi_1(1 - \mu^k)$.

\emph{Part~(b).} Setting $p_k = \pi_1(1 - \mu^k)$:
\[
\mathrm{E}[\nu] = \sum_{k=1}^L p_k
= \pi_1\biggl[L - \sum_{k=1}^L \mu^k\biggr]
= \pi_1\biggl[L - \frac{\mu(1{-}\mu^L)}{1{-}\mu}\biggr].
\]

\emph{Part~(c).} By the Markov property, the joint probability factorises as
\[
\mathrm{P}(\sigma_j {=} 1,\, \sigma_k {=} 1 \mid \sigma_0 {=} 0)
= \mathrm{P}(\sigma_j {=} 1 \mid \sigma_0 {=} 0)\cdot
  \mathrm{P}(\sigma_k {=} 1 \mid \sigma_j {=} 1)
= \pi_1(1{-}\mu^j)(\pi_1 + \pi_0\mu^{k-j}).
\]
Subtracting the product of marginals $\mathrm{E}[\sigma_j]\,\mathrm{E}[\sigma_k] = \pi_1^2(1{-}\mu^j)(1{-}\mu^k)$:
\begin{align*}
\mathrm{Cov}(\sigma_j,\sigma_k)
&= \pi_1(1{-}\mu^j)\bigl[\pi_1 + \pi_0\mu^{k-j} - \pi_1(1{-}\mu^k)\bigr] \\
&= \pi_1(1{-}\mu^j)\bigl[\pi_0\mu^{k-j} + \pi_1\mu^k\bigr]
= \pi_1(1{-}\mu^j)\mu^{k-j}(\pi_0 + \pi_1\mu^j). \qedhere
\end{align*}
\end{proof}

\begin{theorem}[Asymptotic equivalence]\label{thm:asymp-equiv}
For $L \to \infty$, $D_{\mathrm{trans}}(L) = D_\infty + O(1/L)$, with the same limit as the stationary case.
\end{theorem}

\begin{proof}
We show that the transient corrections to both the mean and the variance are bounded, so the ratio converges to the stationary value.

In both regimes $\mathrm{E}[\nu] = \pi_1 L + O(1)$: the transient correction is $-\pi_1\mu(1{-}\mu^L)/(1{-}\mu) = O(1)$. Each transient covariance differs from the stationary value $\pi_1\pi_0\mu^{|k-j|}$ by the factor $(1{-}\mu^j)(\pi_0 + \pi_1\mu^j)/\pi_0$, which converges to~$1$ exponentially as $j \to \infty$. The deviation from the stationary covariance is concentrated at positions $j \lesssim \xi$, where $\xi = 1/\ln(N/t)$ is the correlation length, contributing $O(1)$ to the total variance. Since $\mathrm{Var}_{\mathrm{trans}}(\nu) = \mathrm{Var}_{\mathrm{stat}}(\nu) + O(1)$ and $\mathrm{E}[\nu] = \Theta(L)$,
\[
D_{\mathrm{trans}}(L) = \frac{\mathrm{Var}_{\mathrm{stat}} + O(1)}{\mathrm{E}_{\mathrm{stat}} + O(1)}
= D_\infty + O(1/L). \qedhere
\]
\end{proof}

\begin{question}\label{q:coupling-probabilistic}
Does the coupling parameter $x = N/(2\sqrt{\mathfrak{d}})$ admit a probabilistic interpretation, for instance as a channel capacity of the binary asymmetric channel with transition probabilities $g/N$ and $k/N$?
\end{question}

\section{The Poisson transition and monotonicity}\label{sec:poisson}

\subsection{The Poisson condition}

\begin{theorem}[Poisson condition]\label{thm:poisson}
$D_\infty = 1$ if and only if
\begin{equation}\label{eq:poisson}
\pi_1 = \frac{2\mu}{1{+}\mu}, \qquad\text{equivalently } 2kt = g(g{+}k).
\end{equation}
\end{theorem}

\begin{proof}
Setting $D_\infty = 1$ in Corollary~\ref{cor:dispersion}: $\pi_0(1{+}\mu)/(1{-}\mu) = 1$. Rearranging, $(1{-}\pi_1)(1{+}\mu) = 1{-}\mu$, which gives $\pi_1 = 2\mu/(1{+}\mu)$. To obtain the algebraic form, substitute $\pi_1 = g/(g{+}k)$ and $\mu = t/N = t/(g{+}t{+}k)$, then simplify.
\end{proof}

The Poisson curve $\pi_1 = 2\mu/(1{+}\mu)$ divides the parameter space into an overdispersed region (below: state entries cluster) and an underdispersed region (above: state entries are more regular than Poisson).

For symmetric chains ($g = k$, $\pi_1 = 1/2$), the Poisson condition simplifies to $\mu = 1/3$.

\begin{example}[Carry propagation in addition]\label{ex:carry-dispersion}
For carry propagation, $\pi_1 = 1/2$ and $\mu = 1/p$, which gives
\[
D_\infty = \frac{p{+}1}{2(p{-}1)}.
\]
The Poisson transition $\mu = 1/3$ corresponds to $p = 3$---the same base where the Fibonacci bisection appears.
\end{example}

\begin{table}[ht]
\centering
\caption{Dispersion index for carry propagation in addition.}
\label{tab:dispersion}
\begin{tabular}{cccl}
\toprule
$p$ & $\mu = 1/p$ & $D_\infty$ & Regime \\
\midrule
2 & 1/2 & 3/2 & overdispersed \\
3 & 1/3 & 1 & Poisson \\
5 & 1/5 & 3/4 & underdispersed \\
7 & 1/7 & 2/3 & underdispersed \\
\bottomrule
\end{tabular}
\end{table}

\begin{theorem}[Poisson transition and coupling parameter]\label{thm:poisson-coupling}
For symmetric chains, the Poisson transition lies at $\mu = 1/3$ with coupling parameter
\[
x_{\mathrm{Poisson}} = \frac{3}{2},
\]
independent of~$N$. This is exactly the Chebyshev point where the Fibonacci bisection arises.
\end{theorem}

\begin{proof}
For symmetric chains, $g = k$ forces $\pi_1 = 1/2$. The Poisson condition $\mu = 1/3$ (from Theorem~\ref{thm:poisson}) forces $t/N = 1/3$, hence $g/N = k/N = 1/3$. The coupling parameter is then $x = N/(2\sqrt{gt}) = N/(2\sqrt{(N/3)(N/3)}) = N/(2N/3) = 3/2$, independent of~$N$.
\end{proof}

\subsection{The finite Poisson transition point}\label{sec:monotonicity}

For symmetric chains ($g = k$) with $\sigma_0 = 0$, define the finite Poisson transition point $\mu^*(L)$ as the unique solution of $D_{\mathrm{trans}}(L, \mu) = 1$.

\begin{theorem}[Monotonicity]\label{thm:monotonicity}
For symmetric chains:

\emph{(a)} $D_{\mathrm{trans}}(L, \mu)$ is strictly increasing in~$L$ for each fixed $\mu \in (0,1)$.

\emph{(b)} $D_{\mathrm{trans}}(L, \mu)$ is strictly increasing in~$\mu$ for each fixed $L \ge 1$.

\emph{(c)} $\mu^*(L)$ is strictly decreasing in~$L$ with $\mu^*(L) = 1/3 + 1/(6L) + O(1/L^2)$.
\end{theorem}

The proof of this theorem occupies the remainder of this section. The main tool is the \emph{marginal dispersion}, which reduces the monotonicity in~$L$ to a comparison between the current dispersion index and the contribution of the next position.

\subsection{Marginal dispersion}

The \emph{marginal dispersion} $d_{L+1} := \delta_V(L{+}1)/\delta_E(L{+}1)$ measures the dispersion contribution of position $L{+}1$ relative to the expected contribution.

\begin{lemma}[Closed form]\label{lem:marginal}
$d_{L+1} = \tfrac{1}{2}(1 + \mu^{L+1}) + \mu(1 - \mu^L)/(1 - \mu)$.
\end{lemma}

\begin{proof}
For symmetric chains ($g = k$, $\pi_1 = 1/2$), the marginal variance increment at position $L{+}1$ is
\[
\delta_V(L{+}1) = \mathrm{Var}(\sigma_{L+1}) + 2\sum_{j=1}^{L}\mathrm{Cov}(\sigma_j,\sigma_{L+1}).
\]
Using $\mathrm{Var}(\sigma_{L+1}) = p_{L+1}(1 - p_{L+1})$ with $p_k = \tfrac{1}{2}(1 - \mu^k)$, and $\mathrm{Cov}(\sigma_j,\sigma_{L+1}) = \tfrac{1}{2}(1 - \mu^j)\mu^{L+1-j}(\tfrac{1}{2} + \tfrac{1}{2}\mu^j)$ from Proposition~\ref{prop:transient}(c), the marginal expectation increment is $\delta_E(L{+}1) = p_{L+1} = \tfrac{1}{2}(1 - \mu^{L+1})$. Dividing $\delta_V$ by $\delta_E$ and simplifying the geometric sums yields the stated formula.
\end{proof}

\begin{lemma}[Factorisation]\label{lem:factorisation}
$d_k = D_\infty(\mu)\cdot(1 - \mu^k)$.
\end{lemma}

\begin{proof}
Expanding the formula from Lemma~\ref{lem:marginal} and collecting terms:
\[
d_k = \frac{1{+}\mu}{2(1{-}\mu)} - \mu^k \cdot \frac{1{+}\mu}{2(1{-}\mu)} = D_\infty(1 - \mu^k). \qedhere
\]
\end{proof}

\begin{corollary}[Weighted representation]\label{cor:weighted}
$D(L, \mu) = D_\infty(\mu) \cdot \sum_{k=1}^L (1{-}\mu^k)^2 \big/ \sum_{k=1}^L (1{-}\mu^k)$.
\end{corollary}

\begin{lemma}\label{lem:monotone-marginal}
$d_{L+1}$ is strictly increasing in~$L$, with $d_{L+2} - d_{L+1} = \mu^{L+1}(1{+}\mu)/2 > 0$ and $d_{L+1} \to D_\infty$ from below.
\end{lemma}

\subsection{Proof of Theorem~\ref{thm:monotonicity}(a): Monotonicity in $L$}

\begin{proof}
The idea is to show that $D(L) < d_{L+1}$ for all $L \ge 1$, which implies $D(L) < D(L{+}1)$ since $D(L{+}1)$ is a weighted average of $D(L)$ and $d_{L+1}$. We proceed by induction.

\emph{Base case ($L = 1$).} $D(1) = (1{+}\mu)/2$ and $d_2 = (1{+}\mu)^2/2$, so $d_2 - D(1) = (1{+}\mu)\mu/2 > 0$.

\emph{Inductive step.} Suppose $D(L) < d_{L+1}$. The dispersion $D(L{+}1)$ is a convex combination of $D(L)$ and $d_{L+1}$ (weighted by the cumulative and marginal expectations, respectively). Since $D(L) < d_{L+1}$, it follows that $D(L) < D(L{+}1) < d_{L+1}$. By Lemma~\ref{lem:monotone-marginal}, $d_{L+1} < d_{L+2}$, so the inductive hypothesis holds for $L{+}1$.
\end{proof}

\subsection{Proof of Theorem~\ref{thm:monotonicity}(b): Monotonicity in $\mu$}

\begin{proof}
We express $D(L)$ as a weighted mean and differentiate. Writing $D(L) = \langle d_k \rangle_w$ for the weighted mean with weights $w_k \propto 1 - \mu^k$, the product rule gives
\[
\frac{\partial D}{\partial \mu} = \bigl\langle \tfrac{\partial d_k}{\partial \mu} \bigr\rangle_w + \mathrm{Cov}_w(d_k, \rho_k),
\]
where $\rho_k = (\partial w_k/\partial\mu)/w_k$. We show both terms are non-negative.

\begin{lemma}\label{lem:dk-increasing}
$\partial d_k/\partial\mu > 0$ for all $k \ge 1$ and $\mu \in (0,1)$.
\end{lemma}

\begin{proof}
From $d_k = D_\infty(1 - \mu^k)$ and the factorisation in Lemma~\ref{lem:factorisation}:
\[
\frac{\partial d_k}{\partial\mu} = \frac{1}{1{-}\mu}\biggl[\sum_{j=0}^{k-1}\mu^j - \frac{k}{2}(\mu^{k-1} + \mu^k)\biggr].
\]
By the AM-GM inequality, $\sum_{j=0}^{k-1} \mu^j \ge k\mu^{(k-1)/2}$, and for $\mu \in (0,1)$ we have $\mu^{(k-1)/2} > (\mu^{k-1} + \mu^k)/2$, so the expression in brackets is positive.
\end{proof}

\begin{lemma}\label{lem:covariance}
$\mathrm{Cov}_w(d_k, \rho_k) \ge 0$.
\end{lemma}

\begin{proof}
The coefficient $\rho_k = -k\mu^{k-1}/(1 - \mu^k)$ is strictly increasing in~$k$ (i.e.\ $|\rho_k|$ is decreasing), since $\sum_{j=1}^k \mu^j < k$. Both $d_k$ and $\rho_k$ are increasing in~$k$, so the \v{C}eby\v{s}\"ev sum inequality~\cite{chebyshev-inequality} gives $\mathrm{Cov}_w \ge 0$.
\end{proof}

Both terms are positive, hence $\partial D/\partial\mu > 0$.
\end{proof}

\subsection{Proof of Theorem~\ref{thm:monotonicity}(c): Convergence rate}

\begin{proof}
Existence and uniqueness of $\mu^*(L)$ follow from the intermediate value theorem: $D(L,0) = 1/2 < 1$ and $D(L,\mu) \to \infty$ as $\mu \to 1^-$, and the function is continuous and strictly increasing in~$\mu$ by part~(b). At $\mu = \mu^*(L)$: $D(L{+}1, \mu^*(L)) > 1$ by part~(a), so $\mu^*(L{+}1) < \mu^*(L)$ by part~(b). Convergence to~$1/3$ follows from $D(L,\mu) \to D_\infty(\mu)$ and $D_\infty(1/3) = 1$.

We now determine the rate of convergence. From Corollary~\ref{cor:weighted},
\[
D(L, \mu)
= D_\infty(\mu)\,\frac{S_2(L,\mu)}{S_1(L,\mu)},
\qquad
S_1 = \sum_{k=1}^L (1{-}\mu^k),
\quad
S_2 = \sum_{k=1}^L (1{-}\mu^k)^2.
\]
For fixed $\mu \in (0,1)$ and $L \to \infty$, the geometric tails converge, giving
\begin{align*}
S_1 &= L - \frac{\mu}{1{-}\mu} + O(\mu^L), \\
S_2 &= L - \frac{2\mu}{1{-}\mu} + \frac{\mu^2}{1{-}\mu^2} + O(\mu^L).
\end{align*}
Write $A = \mu/(1{-}\mu)$ and
$B = 2\mu/(1{-}\mu) - \mu^2/(1{-}\mu^2)$.  Then
\[
\frac{S_2}{S_1}
= \frac{L - B}{L - A}
= 1 - \frac{B{-}A}{L} + O(1/L^2),
\]
with
\[
B - A
= \frac{\mu}{1{-}\mu} - \frac{\mu^2}{1{-}\mu^2}
= \frac{\mu}{(1{-}\mu)(1{+}\mu)}.
\]
Hence
\begin{equation}\label{eq:D-expansion}
D(L, \mu) = D_\infty(\mu)\biggl[1 - \frac{\mu}{L(1{-}\mu)(1{+}\mu)} + O(1/L^2)\biggr].
\end{equation}

Now apply the implicit function theorem.  Set
$\Phi(L,\mu) = D(L,\mu) - 1$.  At the asymptotic Poisson
point $\mu_0 = 1/3$ we have $D_\infty(1/3) = 1$ and
\[
\frac{\partial D_\infty}{\partial\mu}\bigg|_{\mu=1/3}
= \frac{1}{(1{-}\mu)^2}\bigg|_{\mu=1/3}
= \frac{9}{4} > 0.
\]
Substituting $\mu^*(L) = 1/3 + \epsilon$ into~\eqref{eq:D-expansion}
and solving $D(L,\mu^*) = 1$ to first order:
\[
\biggl(1 + \frac{9}{4}\epsilon + O(\epsilon^2)\biggr)
\biggl(1 - \frac{3}{8L} + O(1/L^2)\biggr) = 1,
\]
where $\mu/((1{-}\mu)(1{+}\mu))\big|_{\mu=1/3} = 3/8$.
Expanding and solving for~$\epsilon$:
\[
\epsilon = \frac{3}{8L}\cdot\frac{4}{9} + O(1/L^2)
= \frac{1}{6L} + O(1/L^2). \qedhere
\]
\end{proof}

\begin{table}[ht]
\centering
\caption{Convergence of the finite Poisson transition point.}
\label{tab:convergence}
\begin{tabular}{ccccc}
\toprule
$L$ & $\mu^*(L)$ & $\mu^*(L) - 1/3$ & $L\cdot(\mu^*(L) - 1/3)$ & prediction~$1/6$ \\
\midrule
5 & 0.3792 & $4.58 \times 10^{-2}$ & 0.229 & 0.167 \\
10 & 0.3525 & $1.92 \times 10^{-2}$ & 0.192 & 0.167 \\
20 & 0.3422 & $8.90 \times 10^{-3}$ & 0.178 & 0.167 \\
50 & 0.3368 & $3.42 \times 10^{-3}$ & 0.171 & 0.167 \\
100 & 0.3350 & $1.69 \times 10^{-3}$ & 0.169 & 0.167 \\
\bottomrule
\end{tabular}
\end{table}

The residual discrepancy at finite~$L$ is due to the~$O(1/L^2)$ correction.

\begin{question}\label{q:poisson-fibonacci}
The coincidence $x_{\mathrm{Poisson}} = 3/2 = x_{\mathrm{Fibonacci}}$ follows algebraically from the coupling parameter (Theorem~\ref{thm:poisson-coupling}).  Does this identity admit a structural explanation---for instance, by identifying $D_\infty = 1$ as a phase boundary in a nearest-neighbour lattice model?
\end{question}

\section{Generalisation to larger state spaces}\label{sec:generalisation}

\subsection{Setup}

Let $S$ be a finite state space with $|S| = s \ge 2$ and $X$ a finite alphabet with $|X| = N$. Each $x \in X$ defines a transition $T_x\colon S \to S$. Fix a \emph{forbidden state} $s^* \in S$ and set $S' = S \setminus \{s^*\}$. A sequence is \emph{$s^*$-avoiding} if $\sigma_k \in S'$ for all $1 \le k \le L$.

For $|S| = 2$ with $s^* = 1$, pure state avoidance is stricter than the cascade-free condition. For $|S| > 2$ the state avoidance problem acquires a richer structure.

\subsection{The $(s{-}1)$-dimensional transfer matrix}

For $(i,j) \in S' \times S'$, define $n_{ij} = |\{x \in X : T_x(i) = j\}|$. The \emph{restricted transfer matrix} $\widetilde{T}$ has entries $\widetilde{T}_{ij} = n_{ij}$.

\begin{theorem}[State-avoiding count]\label{thm:general}
The number of $s^*$-avoiding sequences of length~$L$ from $\sigma_0 \in S'$ is
\[
a(L, \sigma_0) = e_{\sigma_0}^{\mathsf{T}} \widetilde{T}^L \mathbf{1}.
\]
The sequence satisfies a linear recurrence of order~$s{-}1$ with coefficients given by the elementary symmetric functions of the eigenvalues of~$\widetilde{T}$.
\end{theorem}

The generating function $F(z) = e_{\sigma_0}^{\mathsf{T}}(I - z\widetilde{T})^{-1}\mathbf{1}$ is rational with denominator degree~$s{-}1$.

\subsection{The case $|S| = 3$: Chebyshev structure persists}

For $|S| = 3$ the restricted matrix $\widetilde{T}$ is $2 \times 2$, and the recurrence is
\[
a(L) = \Tr(\widetilde{T})\,a(L{-}1) - \Det(\widetilde{T})\,a(L{-}2).
\]

\begin{theorem}[Chebyshev representation for $|S| = 3$]\label{thm:chebyshev3}
For $|S| = 3$ and $\Det(\widetilde{T}) > 0$,
\[
a(L) = (\sqrt{\Det(\widetilde{T})}\,)^L \cdot U_L(\tilde{x}),
\]
where $\tilde{x} = \Tr(\widetilde{T})/(2\sqrt{\Det(\widetilde{T})})$.
\end{theorem}

So the Chebyshev representation is not specific to the \GEN/\PROP/\KILL{} setup: it appears whenever the restricted transfer matrix is $2\times 2$.

\begin{remark}[Degenerate case: the sediment model on $\ZZ/p\ZZ$]\label{rem:sediment}
Not every state avoidance problem produces non-trivial recurrences. For $S = \ZZ/p\ZZ$, $X = (\ZZ/p\ZZ)^2$, $T_{(a,b)}(\sigma) = \sigma + a + b \bmod p$, the restricted matrix is $\widetilde{T} = p\,J_{p-1}$ (all-ones matrix scaled by~$p$), with dominant eigenvalue $p(p{-}1)$ and all others zero. The count reduces to the geometric sequence $a(L) = (p(p{-}1))^L$.
\end{remark}

\subsection{Ternary three-number addition}\label{sec:ternary}

We illustrate the theory with a concrete example. Adding three digits $a_k + b_k + c_k + \sigma_{k-1}$ in base $p = 3$ produces sums in $\{0,\ldots,8\}$ with carry state $\sigma \in \{0,1,2\}$. The forbidden state is $s^* = 2$ (no double carry).

To construct the restricted transfer matrix, we enumerate the transitions from each non-forbidden state. The alphabet is $X = \{0,1,2\}^3$ with $N = 27$.

\emph{From $\sigma = 0$:} the sum $a + b + c$ ranges over $\{0,\ldots,6\}$. It is $\le 2$ (carry~$0$) for 10 triples, $\in \{3,4,5\}$ (carry~$1$) for 16 triples, and $= 6$ (carry~$2$) for 1 triple. Since carry~$2$ leads to $s^* = 2$, we record $\widetilde{T}_{00} = 10$ and $\widetilde{T}_{01} = 16$.

\emph{From $\sigma = 1$:} the sum $a + b + c + 1$ ranges over $\{1,\ldots,7\}$. It is $\le 2$ (carry~$0$) for 4 triples, $\in \{3,4,5\}$ (carry~$1$) for 19 triples, and $\ge 6$ (carry~$2$) for 4 triples. Hence $\widetilde{T}_{10} = 4$ and $\widetilde{T}_{11} = 19$.

The restricted matrix is therefore
\[
\widetilde{T} = \begin{pmatrix} 10 & 16 \\ 4 & 19 \end{pmatrix},
\]
with $\Tr = 29$, $\Det = 126$, eigenvalues $(29 \pm \sqrt{337})/2$, and coupling parameter $\tilde{x} = 29/(2\sqrt{126}) \approx 1.292$.

The coupling parameter drops from $x = 3/2$ (two-number addition) to $\tilde{x} \approx 1.292$ (three-number addition), reflecting the stronger state coupling when more summands are involved.

\subsection{Binary four-number addition: beyond the Chebyshev threshold}\label{sec:binary4}

Adding four binary digits $a_k + b_k + c_k + d_k + \sigma_{k-1}$ with $\sigma \in \{0,1,2,3\}$, alphabet $X = \{0,1\}^4$ ($N = 16$), and forbidden state $s^* = 3$ (no triple carry). The carry at each position is $\lfloor(a_k + b_k + c_k + d_k + \sigma_{k-1})/2\rfloor$. Since $a + b + c + d + \sigma_{\mathrm{prev}} \le 4 + 2 = 6$, the carry reaches~$3$ only from state $\sigma = 2$ when all four digits equal~$1$.

We construct the $3 \times 3$ restricted matrix on $S' = \{0,1,2\}$ by counting, for each $(i,j) \in S' \times S'$, the number of digit quadruples $(a,b,c,d) \in \{0,1\}^4$ such that $\lfloor(a+b+c+d+i)/2\rfloor = j$.

\emph{From $\sigma = 0$:} the sum $a+b+c+d \in \{0,\ldots,4\}$ has carry $\lfloor s/2 \rfloor$. The counts are: $s = 0$ ($\binom{4}{0} = 1$ way, carry 0), $s = 1$ ($\binom{4}{1} = 4$, carry 0), $s = 2$ ($\binom{4}{2} = 6$, carry 1), $s = 3$ ($\binom{4}{3} = 4$, carry 1), $s = 4$ ($\binom{4}{4} = 1$, carry 2). Hence $\widetilde{T}_{00} = 5$, $\widetilde{T}_{01} = 10$, $\widetilde{T}_{02} = 1$.

\emph{From $\sigma = 1$:} the sum $a+b+c+d+1 \in \{1,\ldots,5\}$ with carry $\lfloor(s+1)/2\rfloor$. Counting: $s = 0$ (carry 0, 1 way), $s = 1$ (carry 1, 4 ways), $s = 2$ (carry 1, 6 ways), $s = 3$ (carry 2, 4 ways), $s = 4$ (carry 2, 1 way). Hence $\widetilde{T}_{10} = 1$, $\widetilde{T}_{11} = 10$, $\widetilde{T}_{12} = 5$.

\emph{From $\sigma = 2$:} the sum $a+b+c+d+2 \in \{2,\ldots,6\}$ with carry $\lfloor(s+2)/2\rfloor$. Counting: $s = 0$ (carry 1, 1 way), $s = 1$ (carry 1, 4 ways), $s = 2$ (carry 2, 6 ways), $s = 3$ (carry 2, 4 ways), $s = 4$ (carry 3, 1 way). Since carry~$3$ is forbidden, $\widetilde{T}_{20} = 0$, $\widetilde{T}_{21} = 5$, $\widetilde{T}_{22} = 10$.

The resulting restricted matrix is
\[
\widetilde{T} = \begin{pmatrix} 5 & 10 & 1 \\ 1 & 10 & 5 \\ 0 & 5 & 10 \end{pmatrix},
\]
with characteristic polynomial $\lambda^3 - 25\lambda^2 + 165\lambda - 280 = 0$, which is irreducible over~$\QQ$ (verified by the rational root test: the potential rational roots $\pm 1, \pm 2, \pm 4, \pm 5, \pm 7, \pm 8, \pm 10, \pm 14, \pm 20, \pm 28, \pm 35, \pm 40, \pm 56, \pm 70, \pm 140, \pm 280$ are all non-roots). The count satisfies the third-order recurrence
\[
a(L) = 25\,a(L{-}1) - 165\,a(L{-}2) + 280\,a(L{-}3),
\]
with initial values $a(0) = 1$, $a(1) = 16$, $a(2) = 255$. No Chebyshev representation exists, since $\dim(\widetilde{T}) = 3 > 2$. The universality criterion (Theorem~\ref{thm:gen-univ}) now requires matching three invariants---trace, sum of $2\times 2$ minors, and determinant---rather than two.

\subsection{Generalised universality}

\begin{theorem}[Generalised universality]\label{thm:gen-univ}
Two stateful operations with state space~$|S| = s$ produce identical $s^*$-avoiding sequences if and only if their restricted transfer matrices share the same characteristic polynomial.
\end{theorem}

\begin{proof}
\emph{Sufficiency.} By Theorem~\ref{thm:general}, the $s^*$-avoiding count from initial state $\sigma_0$ is determined by $\widetilde{T}^L$, whose entries satisfy a linear recurrence with characteristic polynomial equal to that of~$\widetilde{T}$ (by the Cayley--Hamilton theorem). The recurrence coefficients are the elementary symmetric functions of the eigenvalues, which are determined by the characteristic polynomial.  It remains to check the initial conditions: $a(0) = 1$ in both cases, and for $1 \le j \le s{-}2$ the values $a(j)$ are determined by the first $j$ powers of~$\widetilde{T}$, whose traces depend only on the power sums of the eigenvalues, hence on the coefficients of the characteristic polynomial (via Newton's identities).

\emph{Necessity.} If the sequences agree for all~$L \ge 0$, the recurrence coefficients must coincide (since the first $s{-}1$ values determine the recurrence uniquely), and hence the characteristic polynomials are equal.
\end{proof}

For $|S| = 3$ this requires matching two invariants (trace and determinant); for $|S| = s$ it requires $s{-}1$ invariants. The Chebyshev representation persists precisely when $\dim(\widetilde{T}) = 2$, i.e.\ for $|S| \in \{2, 3\}$.

\subsection{Applicability of the \GEN/\PROP/\KILL{} decomposition}\label{sec:classification}

\begin{theorem}\label{thm:applicability}
The transfer matrix method with \GEN/\PROP/\KILL{} decomposition applies if and only if $|S| = 2$ and every transition $T_x$ is either constant or the identity.
\end{theorem}

\begin{proof}
The \GEN/\PROP/\KILL{} classification exhausts all maps $\{0,1\} \to \{0,1\}$ that are either constant ($T_x \equiv 0$ or $T_x \equiv 1$) or the identity ($T_x = \mathrm{id}$); the only remaining map on $\{0,1\}$ is the negation $\sigma \mapsto 1 - \sigma$, which is excluded by assumption.  For $|S| = 2$ without negation, every symbol falls into exactly one of \GEN, \PROP, \KILL, so the decomposition is valid and the $2 \times 2$ transfer matrix of Section~\ref{sec:transfer} governs the count.  Conversely, if $|S| > 2$, there exist transition maps that are neither constant nor the identity (e.g.\ permutations of three or more states), and a three-class decomposition cannot capture the transition structure.
\end{proof}

So the cascade-free condition ($P^* = 0$) marks the boundary between stateless and genuinely stateful arithmetic; $\rho(L)$ and~$x$ tell us how close a particular operation is to that boundary.

\section*{Declaration of competing interests}

The author declares that there are no known competing financial interests or personal relationships that could have appeared to influence the work reported in this paper.

\section*{Funding}

This research did not receive any specific grant from funding agencies in the public, commercial, or not-for-profit sectors.

\section*{Data availability}

No data was used for the research described in this article.

\section*{Declaration of generative AI and AI-assisted technologies in the manuscript preparation process}

During the preparation of this work the author used Claude (Anthropic) in order to assist with structuring the manuscript, verifying computational examples, and improving the exposition. All mathematical content, proofs, and results were developed by the author. After using this tool, the author reviewed and edited all content and takes full responsibility for the content of the article.


\appendix
\setcounter{table}{0}
\renewcommand{\thetable}{A.\arabic{table}}
\section{Verification}\label{app:verification}

\subsection{Exhaustive verification}

\begin{table}[ht]
\centering
\caption{Exhaustive verification of the carry--doubling scaling law.}
\label{tab:verification}
\begin{tabular}{cccccc}
\toprule
$p$ & $L$ & $a_{\mathrm{carry}}(L)$ & $a_{\mathrm{dbl}}(L)$ & $p^L \cdot a_{\mathrm{dbl}}(L)$ & Match \\
\midrule
3 & 2 & 72 & 8 & 72 & $\checkmark$ \\
3 & 3 & 567 & 21 & 567 & $\checkmark$ \\
3 & 4 & 4\,455 & 55 & 4\,455 & $\checkmark$ \\
5 & 2 & 575 & 23 & 575 & $\checkmark$ \\
5 & 3 & 13\,125 & 105 & 13\,125 & $\checkmark$ \\
5 & 4 & 299\,375 & 479 & 299\,375 & $\checkmark$ \\
7 & 2 & 2\,254 & 46 & 2\,254 & $\checkmark$ \\
7 & 3 & 103\,243 & 301 & 103\,243 & $\checkmark$ \\
7 & 4 & 4\,727\,569 & 1\,969 & 4\,727\,569 & $\checkmark$ \\
11 & 2 & 14\,036 & 116 & 14\,036 & $\checkmark$ \\
11 & 3 & 1\,625\,151 & 1\,221 & 1\,625\,151 & $\checkmark$ \\
11 & 4 & 188\,151\,491 & 12\,851 & 188\,151\,491 & $\checkmark$ \\
13 & 2 & 27\,547 & 163 & 27\,547 & $\checkmark$ \\
13 & 3 & 4\,484\,077 & 2\,041 & 4\,484\,077 & $\checkmark$ \\
13 & 4 & 729\,876\,355 & 25\,555 & 729\,876\,355 & $\checkmark$ \\
\bottomrule
\end{tabular}
\end{table}

\subsection{The case $p = 2$}

At $p = 2$ the binary doubling has no \PROP{} digit ($2d$ is even for all $d \in \{0,1\}$, so $2d \neq p{-}1 = 1$). Algebraically, \PROP{} exists for doubling if and only if $p$ is odd. The coupling parameters diverge: $x_{\mathrm{carry}} = \sqrt{2}$ but $x_{\mathrm{dbl}} = \infty$ (since $\mathfrak{d}_{\mathrm{dbl}} = 0$).

\end{document}